\documentclass[12pt,reqno]{amsart}
\usepackage{hyperref}
\usepackage{fullpage}
\usepackage{amsthm}
\usepackage[all]{xypic}
\def\m{\mathfrak{m}}
\def\R{\mathcal{R}}
\def\NN{\mathbb{N}}

\def\m{\mathfrak{m}}
\def\R{\mathcal{R}}
\def\F{\mathcal{F}}

\def\H{\mathcal{H}}
\def\S{\mathcal{S}}
\def\NN{\mathbb{N}}
\def\red{\operatorname{red}}

\def\dim{\operatorname{dim}}

\def\red{\operatorname{red}}

\newtheorem{theorem}{Theorem}[section]
\newtheorem{definition}[theorem]{Definition}

\newtheorem{proposition}[theorem]{Proposition}
\newtheorem{example}[theorem]{Example}

\newtheorem{question}[theorem]{Question}
\newtheorem{remark}[theorem]{Remark}
\newtheorem{corollary}[theorem]{Corollary}

\begin{document}
\title{On the Vasconcelos inequality for the fiber multiplicity of modules}

\author[Balakrishnan R.]{Balakrishnan R.$^*$}
\email{r.balkrishnan@gmail.com}

\author{A. V. Jayanthan}
\email{jayanav@iitm.ac.in}
\address{Department of Mathematics, Indian Institute of Technology
Madras, Chennai, INDIA -- 600036.}
\thanks{$*$ Supported by the Council of Scientific and Industrial
Research (CSIR), India}
\thanks{AMS Classification 2010: 13D40, 13A30}
\maketitle

\begin{abstract}
Let $(R,\m)$ be a Noetherian local ring of dimension $d>0$ with
infinite residue field. Let $M$ be a finitely generated proper
$R$-submodule of a free $R$-module $F$ with $\ell (F/M) < \infty$ and
having rank $r$. In this article, we study the fiber multiplicity
$f_0(M)$ of the module $M$. We prove that if $(R,\m)$ is a two
dimensional Cohen-Macaulay local ring, then $f_0(M)\le
br_1(M)-br_0(M)+\ell (F/M)+\mu(M)-r$, where $br_i(M)$ denotes the
$i^{th}$ Buchsbaum-Rim coefficient of $M$.
\end{abstract}

\section{Introduction}
Throughout the paper, we will assume that $(R,\m)$ is a Noetherian
local ring of dimension $d>0$ with infinite residue field and
$M$ is a finitely generated proper submodule of a free $R$-module $F$
with $\ell(F/M) < \infty$ and having rank $r$.  Let $\mathcal
S(F)=\underset {n\ge
0}{\bigoplus}{\mathcal S_n(F)}$ denote the Symmetric algebra of $F$,
and $\mathcal R(M)=\underset {n\ge 0}{\bigoplus}\mathcal R_n(M)$
denotes the Rees algebra of $M$, which is image of the natural map from
the Symmetric algebra of $M$ to the Symmetric algebra of $F$.
Generalizing the notion of Hilbert-Samuel function, D. A. Buchsbaum
and D. S. Rim studied the function $BF_M(n)=\ell (\mathcal
S_n(F)/\mathcal R_n(M))$ for $n\in \mathbb{N}$. In \cite {BR}, they
proved that $BF_M(n)$ is given by a polynomial of degree $d+r-1$ for
$n\gg0$, i.e., there exists a polynomial $BP_M(x)\in \mathbb Q[x]$ of
degree $d+r-1$ such that $BF_M(n)=BP_M(n)$ for $n\gg 0.$ The function
$BF_M(n)$ is called the Buchsbaum-Rim function of $M$ with respect to
$F$ and the polynomial $BP_M(n)$ is called the corresponding
Buchsbaum-Rim polynomial. Following the notation used for the
Hilbert-Samuel polynomial, one writes the Buchsbaum-Rim polynomial as
$$BP_M(n)=\sum_{i=0}^{d+r-1}(-1)^i br_i(M){n+d+r-i-2\choose
d+r-i-1}.$$ The coefficients $br_i(M)$ for $i=0,\ldots ,d+r-1$ are
known as Buchsbaum-Rim coefficients. For basic properties of the
Buchsbaum-Rim function and the Buchsbaum-Rim polynomial, we refer the
reader to \cite{hs},\cite{WV}. In this article we study the fiber
multiplicity $f_0(M)$ of the module $M$ and relate it with $br_0(M)$
and $br_1(M)$.

Let $\F(M) := \R(M) \otimes R/\m$ denote the fiber cone of $M$.  In
Section 2, we study Cohen-Macaulayness of fiber cone $\F(M)$. The
Cohen-Macaulayness of $\F(I)$, where $I$ is an ideal in $R$, has been
of interest and has been studied widely, see for example \cite
{clare},\cite {Goto},\cite {jayan2},\cite {shah}. In \cite {shah}, K.
Shah studied the Hilbert function and the Cohen-Macaulayness of
$\F(I)$.
\begin{theorem}\cite [Theorem 1]{shah}\label{cmfibercone}
Let $(R,\m)$ be a local ring. Suppose $I$ is an ideal which is
integral over a regular sequence $\underline{x}$ such that
$I^2=I\underline{x}$. Then $\F(I)$ is Cohen-Macaulay.
\end{theorem}
We study some basic properties of the fiber cone $\F(M)$. We give a
characterization for the Cohen-Macaulayness of $\F(M)$. We then prove
an analogue of Theorem \ref{cmfibercone} in the case of modules over
two dimensional Cohen-Macaulay local rings. We first recall some
basics on reduction of modules.

Let  $N$ be a submodule of $M$. We say that $N$ is reduction
of $M$ if Rees algebra $\mathcal R(M)$ is integral over the
$R$-subalgebra $\mathcal{R}(N)$. Equivalently, there exists $n_0$ such
that $\mathcal R_{n+1}(M)=N\mathcal R_n(M)$ for $n\geq n_0$,
where the multiplication is done as $R$-submodules of $\R(M)$. The
least integer $s$ such that $\mathcal R_{s+1}(M)=N\mathcal R_s(M)$ is
called the reduction number of $M$ with respect to $N$, denoted as
$\red_N(M)$. The reduction number of the module $M$, denoted $\red
(M),$ is defined as $\red (M)=\min \{\red_N(M):~ N \mbox{ is a minimal
reduction of M}\}$. If $N$ is a submodule of $F$ generated by $d+r-1$
elements such that $\ell(F/N) < \infty$, then $N$ is said to be a
parameter module. It was proved in \cite{BUV} that if $\ell(F/M) <
\infty$, then there exists minimal reduction generated by $d+r-1$
elements.  For more details on minimal reductions, we refer the reader
to \cite{hs} and \cite{WV}. In this article, we prove:

\begin{theorem}
Let $(R,\m)$ be a $2$-dimensional Cohen-Macaulay local ring.
Let $M \subset F$ be such that $\ell (F/M) < \infty$ and having rank
$r$. If $\red (M)\leq 1$, then $\F(M)$ is Cohen-Macaulay.
\end{theorem}

In \cite {vas2}, A. Corso, C. Polini and W. Vasconcelos studied the
multiplicity of the fiber cone $\F(I)$. One of the main result
obtained in \cite {vas2} is an inequality relating fiber multiplicity
$f_0(I)$, Hilbert coefficients $e_0(I)$ and $e_1(I)$ and some other
invariants of $I$:
\begin{theorem}\cite [Theorem 2.1]{vas2} \label{cpv-ideal}
Let $(R,\m)$ be a Cohen-Macaulay local ring of dimension $d>0$ with
infinite residue field. Let $I$ be an $\m$-primary ideal. Then we have
that $$f_0(I)\le  e_1(I)-e_0(I) + \ell(R/I) + \mu(I) - d + 1,$$
where $\mu(M)$ denotes the cardinality of a minimal generating set of
an $R$-module $M$.
\end{theorem}
Motivated by this inequality, Vasconcelos raised the question:
\begin{question}\label{vasco-qn}
Let $(R,\m)$ be a $d$-dimensional Cohen-Macaulay local ring and
$M\subset F$  with $\ell (F/M)<\infty$ and having rank $r$. Then, does
the fiber multiplicity $f_0(M)$ satisfy
$$f_0(M)\le br_1(M)-br_0(M)+\ell(F/M)+\mu (M)-d-r+2?$$
\end{question}
In Sections 3  and 4, we address the above question. In Section
3, we prove that Question \ref{vasco-qn} has an affirmative answer
when $\dim R = 2$. In Section 4, we establish the inequality for 
modules of the form $M=I\oplus \cdots \oplus I \oplus J \oplus \cdots
\oplus J$, where $I$ is an $\m$-primary ideal in $R$ and $J$ is a 
minimal reduction of $I$.  For $d=1$, we provide a counter example.
\vskip 2mm
\noindent
\textbf{Acknowledgements:} We sincerely thank the referees for pointing out several errors, some of them typographical and some of them mathematical, which tremendously improved the exposition.
\section{Fiber cone of modules}
In this section, we study Cohen-Macaulay property of fiber cone of
modules. We begin by recalling the definition of fiber cone $\F(M)$.
\begin{definition}
Let $M\subset F$ be such that $\ell(F/M) < \infty$ and having rank
$r$. The fiber cone of $M,$ denoted by $\F(M),$ is defined as
$$\F(M):=\R(M) \otimes R/\m= \bigoplus \frac{\R_i(M)}{\m \R_i(M)},$$
where $\R(M)$ is the Rees algebra of $M$.
\end{definition}
The Krull dimension of $\F(M)$ is known as the analytic spread of $M$
and is equal to $d+r-1$, \cite{BUV}. The Hilbert function of
$\F(M)$  is given by
$$H(\F(M),n)=\ell (\R_n(M)/\m \R_n(M)), \mbox{ for } n\in \NN .$$
The corresponding Hilbert polynomial, of degree $d+r-2$ for $n\gg 0$,
is written as
$$H(\F(M),n)=f_0(M){n+d+r-2\choose d+r-2}-f_1(M){n+d+r-3\choose
d+r-3}+\cdots +(-1)^{d+r-2}f_{d+r-2}(M).$$
The leading coefficient $f_0(M)$ is called the fiber multiplicity of
$M$. Let $\H(\F(M),t)$ denote the Hilbert series of $\F(M),$ i.e.,
$$\H(\F(M),t)=\sum_{n=0}^{\infty}H(\F(M),n)t^n.$$\\
We now give a characterization for the Cohen-Macaulayness of the fiber
cone of a module in terms of its Hilbert series and its fiber
multiplicity. See also \cite [Proposition 8.40]{WV}. We skip the proof
of Theorem \ref{cmfm} as it is routine.

\begin{theorem}\label{cmfm}
Let $(R,\m)$ be a Noetherian local ring, $M\subset F$ be such that
$\ell(F/M)< \infty$, having rank $r$ and $N\subseteq M$ be a minimal
reduction. Then the following are equivalent:
  \begin{enumerate}
    \item Fiber cone $\F(M)$ is Cohen-Macaulay;
    \item $\H(\F(M),t)=\frac{1}{(1-t)^a} \sum_{i=0}^{b}\ell\left(\frac{\R_i(M)}{N\R_{i-1}(M)+\m \R_i(M)}\right)t^{i}$,\\  ~~~~~ where $b=\red _N(M)$, $a= \dim \F(M)$;
    \item $f_0(M)=\sum_{i=0}^{b} \ell\left(\frac{\R_i(M)}{N\R_{i-1}(M)+\m \R_i(M)}\right).$
  \end{enumerate}
\end{theorem}

We also note that Theorem \ref {cmfm} gives an upper bound on the reduction number of $M$ when $\F(M)$ is Cohen-Macaulay:
\begin{remark}\label{red_bound}
For an $R$-module $K$, let $\mu(K)$ denote the minimum number of
generators of $K$. Note that if $\F(M)$ is Cohen-Macaulay, then
$$f_0(M)=\sum_{i=0}^{b} \ell \left(\frac{\R_i(M)}{N\R_{i-1}(M)+\m
\R_i(M)}\right)=1+\sum_{i=1}^{b}\mu
\left(\frac{\R_i(M)}{N\R_{i-1}(M)}\right).$$ Since $\mu
\left(\frac{\R_i(M)}{N\R_{i-1}(M)}\right)\ge 1$ for all $1\le i\le
b=\red(M)$, we get $$\red(M)\le f_0(M) - 1.$$
Using the fact that $\ell(M/N+\m M) = \mu(M) - \mu(N) = \mu(M) -
(d+r-1)$, we obtain a better bound in the above case:
$$\red(M)\le f_0(M)-\mu (M)+d+r-1.$$
\end{remark}
It is natural to expect that smaller reduction number of the module
$M$ force good properties on $\F(M)$. If $\red (M)=0$, then $M$ is a
parameter module and hence $\F(M)$ is a polynomial ring over the
residue field of $R$. Next natural condition is to consider when $\red
(M)$ is one. We extend K. Shah's, \cite{shah}, result to the case of modules over
two dimensional Cohen-Macaulay rings:
\begin{theorem}\label{cmfm2}
Let $(R,\m)$ be a $2$-dimensional Cohen-Macaulay local ring. Let $M
\subset F$ be such that $\ell (F/M) < \infty$ and having rank $r$. If
$\red (M)\leq 1$, then $\F(M)$ is Cohen-Macaulay.
\end{theorem}
\begin{proof}
Let $N\subseteq M$ be a minimal reduction such that $\red (M)=
\red_N(M)=1$. Now $\mu (N)=r+1$. By Theorem \ref {cmfm}, it is enough
to show that
  $$\H(\F(M),t)=\frac{1}{(1-t)^{r+1}}\left[1+\ell\left(\frac{M}{N+\m M}\right)t\right].$$
Let $\{x_1,\ldots, x_{r+1}\}$ be a minimal generating set for $N$.
Extend this to a minimal generating set $\{x_1,\ldots,
x_{r+1},y_1,\ldots,y_p\}$ for $M$ such that $\mu (M)=r+p+1$. Set
$K=(y_1,\ldots,y_p)\subseteq M$. So $M=N+K$.
Identifying the elements $x_i$'s and $y_j$'s with their images in
$\R_1(M)$, it can be seen that
$\R_n(M)$ is generated as an $R$-module by $\{\R_n(N),\R_{n-1}(N)\R_1(K),\ldots,\R_1(N)\R_{n-1}(K),\R_n(K)\}$.

Since $\red _N(M)=1, \R_1(N)\R_1(M)=\R_2(M)$. This implies that $\R_i(N)\R_1(M)=\R_{i+1}(M)$ for all $i\ge 1$. Now for $i\ge 2$,
\begin{eqnarray*}
  \R_{n-i}(N)\R_i(K)&\subseteq &\R_{n-i}(N)\R_i(M)\subseteq \R_n(M)\\
  &=&\R_{n-1}(N)\R_1(M)\\
  &=&\R_n(N)+\R_{n-1}(N)\R_1(K).
\end{eqnarray*}
Therefore $$\R_n(M)=\langle{\R_n(N),\R_{n-1}(N)\R_1(K)\rangle}.$$ Set
$T_n=\left \{x_1^{\alpha _1}\cdots x_{r+1}^{\alpha _{r+1}}\mid \alpha
_1 +\cdots+\alpha _{r+1}=n\right \}$, where products are taken in
$\R(N)$. Let $\Delta _1,\ldots ,\Delta _{k_{(n)}}$ and $\delta
_1,\ldots,\delta _{k_{(n-1)}}$ denote elements of $T_n$ and $T_{n-1}$
respectively. \\Set $S_n=\{\delta _{i}y_j\mid i=1,\ldots ,
k_{(n-1)},j=1,\ldots, p\}$\\
\textbf{Claim:} $T_n \cup S_n$ is a minimal generating set for $\R_n(M)$.
\\It is clear that $T_n \cup S_n$ generates $\R_n(M)$. We only need to prove the minimality. So let us assume that
$$\sum _{i=1}^{k_{(n)}} r_i \Delta _i + \sum _{i=1}^{k_{(n-1)}}
\sum_{j=1}^p s_{ij}\delta _{i}y_{j}=0.$$
Suppose $s_{i_{_0}j_{_0}}\notin \m$ for some $i_0,j_0$.
Rewriting the above relation, we get
\begin{eqnarray*}
  \sum _{i=1}^{k_{(n)}} r_i \Delta _i + \left(\sum _{j=1}^{p}
  s_{i_{_0}j}y_j\right)\delta _{i_{_0}}+ \sum _{i=1,i\neq
	i_{_0}}^{k_{(n-1)}}\sum_{j=1}^p  s_{ij}\delta _{i}y_{j}&=&0,\\
i.e.,\left(\sum _{j=1}^{p} s_{i_{_0}j}y_j\right)\delta _{i_{_0}}+ \sum
_{i=1,i\neq i_{_0}}^{k_{(n-1)}}\sum_{j=1}^p  s_{ij}\delta _{i}y_{j}&=& -\sum _{i=1}^{k_{(n)}} r_i \Delta _i \in \R_n(N).
\end{eqnarray*}
By  \cite [Corollary 4.5]{HH}, we get $\sum _{j=1}^{p} s_{i_{_0}j}y_j \in N$. Since $s_{i_{_0}j_{_0}}$ is a unit, this implies that $y_{j_{_0}}\in (N,y_1,\ldots ,\hat{y_{j_{_0}}},\ldots,y_p)$ contradicting the minimality of the generating set considered for $M$ above. Therefore $s_{i_{_0}j_{_0}}\in \m$. Therefore,  
\[ \sum _{i=1}^{k_{(n)}} r_i \Delta _i = - \sum _{i=1}^{k_{(n-1)}}\sum_{j=1}^p s_{ij}\delta _{i}y_{j} \in R_n(N) \cap \m R_n(M) = \m R_n(N).\] Therefore $r_i \in \m$ for all $i$ which completes the proof of the claim.\\
Note that $k_{(n)}={n+r\choose r}$ and $\mid T_n \cup S_n \mid = {n+r\choose r}+p{n-1+r\choose r}$. Therefore we have
\begin{eqnarray*}
  \H(\F(M),t)&=&\sum _{n=0}^{\infty} \mu (\R_n(M))t^n\\
  &=&\sum _{n=0}^{\infty}\left[{n+r\choose r}+p{n-1+r\choose r}\right]t^n\\
  &=&\sum _{n=0}^{\infty}{n+r\choose r}t^n + p \sum _{n=0}^{\infty}\left[{n+r\choose r}-{n+r-1\choose r-1}\right]t^n\\
  &=&\frac{1}{(1-t)^{r+1}}+p\left[\frac{1}{(1-t)^{r+1}}-\frac{1}{(1-t)^{r}}\right]\\
  &=&\frac{1+pt}{(1-t)^{r+1}}.
\end{eqnarray*}
Observe that
\begin{eqnarray*}
  \ell \left(\frac{M}{N+\m M}\right)&=&\ell \left(\frac{M}{\m M}\right)-\ell \left(\frac{N+\m M}{\m M}\right)\\
  &=&\mu (M)-\ell \left(\frac{N}{\m M\cap N}\right)\\
  &=&\mu (M)-\ell \left(\frac{N}{\m N}\right)\\
  &=&\mu (M)-\mu (N)\\
  &=&p.
\end{eqnarray*}
Therefore by Theorem \ref{cmfm}, $\F(M)$ is Cohen-Macaulay.
\end{proof}
One of the key ideas used in the above proof is the analytical
independence of the generators of $N$, \cite[Corollary 4.5]{HH}. This
result is proved in dimension $2$ and as far as we know, an analogue
of this result in higher dimensions is not known. Once this result is
generalized to higher dimensions, the above proof goes through for
higher dimensions as well.

At this stage, we would also like to compare the case of ideals with
that of modules. In the case of ideals, the above result, in much more
generality, has a much simpler proof due to the existence of the
associated graded ring and the beautiful and one of the most basic
results on regular sequences, namely Valabrega--Valla theorem. In the
case of modules, both associated graded ring as well as a
Valabrega--Valla type theorem are missing.

Following is an interesting observation on the Cohen-Macaulayness of
$F(M)$ in a special case.
\begin{proposition}\label{I=J}
Let $(R,\m)$ be a $2$-dimensional Cohen-Macaulay local ring and $M=I\oplus J$, where $I$
is an $\m$-primary ideal $J$ be a minimal
reduction of $I$. If $\red (M)\le 1$, then $I=J$.
\end{proposition}
\begin{proof}
Since $\red(M)\le 1$, by Theorem \ref{cmfm2}, $\F(M)$ is Cohen-Macaulay.\\
Now by Theorem \ref{cmfm}, $f_0(M)=\mu (M)-2=\mu (I)+\mu (J)-2=\mu (I)$. The Hilbert function of $\F(M)$ is given by
  \begin{eqnarray*}
    H(\F(M),n)&=&\ell \left(\frac{\R_n(M)}{\m \R_n(M)}\right)\\
    &=&\sum_{i=0}^{n}\ell\left(\frac{J^iI^{n-i}}{\m J^iI^{n-i}}\right)\\
    &=&\ell \left(\frac{I^n}{\m I^n}\right)+\sum_{i=1}^{n-1}\ell \left(\frac{I^n}{\m I^n}\right)+\ell\left(\frac{J^n}{\m J^n}\right)\\
    &=&nH(\F(I),n)+H(\F(J),n).
  \end{eqnarray*}
Therefore $f_0(M)=2f_0(I)$. Since $\red (I)\leq 1,$ $\F(I)$ is
Cohen-Macaulay and hence $f_0(I)=\mu (I)-1$. Hence $\mu (I)=2$ and hence a parameter ideal. Therefore $I=J.$
\end{proof}

\section {Vasconcelos inequality for $d=2$}

In this section, we prove the Vasconcelos inequality for modules
over two dimensional Cohen-Macaulay rings. We adopt the technique used
to prove the inequality for $f_0(I)$ in \cite {vas2}. The idea
involves using the knowledge of the Hilbert polynomial of the Sally
module $S_N(M)$, where $N$ is a minimal reduction of $M$. The notion
of Sally module $S_N(M)$ was introduced in \cite{bala}, extending the
corresponding notion of ideals, \cite{WV1}. We first recall the definition of
$S_N(M)$ from \cite{bala}.
\begin{definition}
Let $(R,\mathfrak m)$ be a Noetherian local ring, $M \subset F$ be
an $R$-module and $N\subseteq M$
be an $R$-submodule. Then the Sally module of $M$ with respect to $N$
is defined as $ S_N(M):=\underset {n\ge 1}\oplus \frac {\mathcal
{R}_{n+1}(M)}{M\mathcal {R}_{n}(N)}.$
\end{definition}

Note that if $N$ is a reduction of $M$, then $S_N(M)$ is a finitely
generated $\R(N)$-module and $S_N(M) = 0$ if and only if $\red_N(M) =
1$.  If $r = 1$, then this definition coincides with the definition of
Sally module of an ideal $I$ with respect to a minimal reduction $J$.
It is to be noted that in this case, $\dim_{\R(J)} S_J(I) = d$ if
$S_J(I) \neq 0$,
\cite{WV}. Note that the proof of this result given by Corso et al.,
\cite[Theorem 2.1]{vas2}, can not be adapted to the
case of modules. Another approach to the dimension is through the
Hilbert function. In the case of rank one, the Hilbert polynomial
of the Sally module can be computed and then using Northcott
inequality along with Huneke-Ooishi theorem one can conclude that the
dimension of the Sally module is $d$. This approach also fails in the
case of modules since an analogue of Northcott inequality (for $d \geq
3$) and Huneke-Ooishi theorem (even for $d \geq 2$) are not known.

Suppose $M \subset F$ is of rank $r$ with $\ell(F/M) < \infty$ and $N$
is a minimal reduction of $M$. We first show that if $S_N(M) \neq 0$, then
 $\dim_{\R(N)} S_N(M) \le d+r-1$.
Set $T=\oplus T_n=\oplus \frac {S_n(F)}{M\R_{n-1}(N)}$ and
$S_N(M)=\oplus S_n =\oplus \frac {R_{n+1}(M)}{M\R_n(M)}$. Let
$BF_M(n)$ and $BF_N(n)$ denote the Buchsbaum-Rim functions of $M$ and
$N$ respectively, i.e., $BF_M(n)=\ell
\left(\frac{S_n(F)}{\R_n(M)}\right)$ and $BF_N(n)=\ell
\left(\frac{S_n(F)}{\R_n(N)}\right)$. Observe that
$$BF_M(n)\le \ell (T_n)\le BF_N(n).$$
Since Buchsbaum-Rim polynomials of $M$ and $N$ are of degree $d+r-1$
with same leading coefficients, it follows that the Hilbert polynomial
of $T$ is of degree $d+r-1$ with the same leading coefficient as that
of Buchsbaum-Rim polynomial of $M$. Note that
$$\ell (S_{n-1})=\ell (T_n) - BF_M(n).$$
Therefore it follows that the Hilbert polynomial of Sally module is of
degree at most $d+r-2$ and hence $\dim S_N(M)\le d+r-1$.  If $d=2$,
then by  \cite[Theorem 3.2]{bala}, it follows that, for $n\gg 0$

\begin{eqnarray}\label{sally-poly}
\ell (S_{n-1}) & =  & [br_1(M)-br_0(M)+\ell (F/M)]{n+r-1\choose
r}-br_2(M){n+r-2\choose r-1} \nonumber \\ & & +\cdots +(-1)^rbr_{r+1}(M).
\end{eqnarray}

\noindent
It is not known whether the equality $br_0(M) - br_1(M) = \ell(F/M)$
implies $\red_N(M) = 1$ and hence we can not possibly conclude from
the above equation that $\dim_{\R(N)} S_N(M) = r$. Therefore, we ask:

\begin{question}\label{sallydim}
Let $(R,\m)$ be a $d$-dimensional Cohen-Macaulay local ring . Let $M
\subset F$ be an $R$-module of rank $r$ with $\ell(F/M) < \infty$ and
$N \subseteq M$ be a minimal reduction of $M$. If $S_N(M)$ is
non-zero, then is $dim_{\R(N)} S_N(M)=d+r-1$?
\end{question}
Keeping in mind the case $r = 1$, we would like to ask another
question, an affirmative answer to which will give an affirmative
answer to Question \ref{sallydim}:
\begin{question}
Let $(R,\m)$ be a $d$-dimensional Cohen-Macaulay local ring. Let $M
\subset F$ be an $R$-module of rank $r$ with $\ell(F/M) < \infty$ and
$N$ be a minimal reduction of $M.$ Is the multiplicity of the Sally
module, $e_0(S_N(M)) = \ell(F/M) +br_1(M) - br_0(M)$?
\end{question}
We now prove the Vasconcelos inequality for modules over $2$-dimensional
Cohen-Macaulay local rings.
\begin{theorem}\label {cpv-general}
Let $(R,\m)$ be a $2$-dimensional Cohen-Macaulay local ring. Let
$M\subset F$ be such that $\ell (F/M)< \infty$ and having rank $r$.
Then $$f_0(M)\le br_1(M)-br_0(M)+\ell(F/M)+\mu
(M)-r.$$ If $\red (M) \leq 1$, then the equality holds.
\end{theorem}

\begin{proof}
Let $N$ be a minimal reduction of $M$. Let us choose  $f_1,\ldots ,
f_k $ from $M$ such that $M=(N,f_1,\ldots , f_k)$, where $k=\mu
(M)-\mu (N)=\mu (M)-r-1$. Let $S_N(M)$ be the Sally module of $M$
with respect to $N$. For $i=1,\ldots, k$, let $g_i$ denote image of
$f_i$ in $\R_1(M)$.
Consider the following $\R(N)$-module homomorphisms
$$i:\R(N)\rightarrow \R(M)\mbox{ and } \phi : \R(N)^{k} \rightarrow \R(M),$$
where $i$ is the natural inclusion map and $\phi$ is defined by $\phi (e_i)=g_i$ for $i=1,\ldots,k$, where $\{e_1,\ldots,e_k\}$ is the standard basis for $\R(N)^k$. Now consider the following graded exact sequence of $\R(N)$-modules
$$\R(N)\oplus \R(N)^k[-1] \stackrel {\psi }\longrightarrow \R(M) \longrightarrow S_N(M)[-1]\longrightarrow 0,$$
where $\psi$ is induced by $i$ and $\phi$. Tensor the above sequence with $-\otimes R/\m$ to get the following graded exact sequence with corresponding induced maps
$$\F(N)\oplus \F(N)^k[-1] \stackrel {\psi }\longrightarrow \F(M) \longrightarrow S_N(M)[-1]\otimes \frac{R}{\m}\longrightarrow 0.$$
Taking lengths of the graded parts we get, for $n\in \NN$,
\begin{eqnarray*}
  \ell ([\F(M)]_n)
  &\le&\ell \left([\F(N)]_n + [\F(N)^k]_{n-1}\right)+\ell \left(\left[\frac{S_N(M)}{\m S_N(M)}\right]_{n-1}\right)\\
  &\le&\ell \left([\F(N)]_n \right) + \ell \left([\F(N)^k]_{n-1}\right)+\ell \left(\left[S_N(M)\right]_{n-1}\right).
\end{eqnarray*}
Note that for $n\gg 0$,

\begin{eqnarray*}
\ell \left([\F(M)]_n \right)&=&\sum _{i=0}^{r}(-1)^if_i(M){n+r-i\choose r-i},\\
\ell \left([\F(N)]_n \right)&=& {n+r\choose r},\\
\ell \left([\F(N)^k]_{n-1}\right)&=& k{n+r-1\choose r},\\
\ell
\left(\left[S_N(M)\right]_{n-1}\right)&=&\left[br_1(M)-br_0(M)+\ell
  (F/M)\right]{n+r-1\choose r}\\
& & + \sum _{i=1}^{r}(-1)^{i}br_{i+1}(M){n+r-1-i\choose r-i},
\end{eqnarray*}
where the last equality follows from (\ref{sally-poly}).
It follows, by comparing the leading coefficients, that
\begin{eqnarray*}
  f_0(M)
  &\le&1+k+br_1(M)-br_0(M)+\ell (F/M)\\
  &=&br_1(M)-br_0(M)+\ell (F/M)+\mu (M)-r.
\end{eqnarray*}
Now assume that $\red (M)=1$. It follows from Theorem $\ref {cmfm2}$
that $\F(M)$ is Cohen-Macaulay and hence $f_0(M)=1+\mu (M)-\mu (N)=\mu
(M)- r.$ Since $\red(M) = 1$, by \cite[Theorem 3.3]{bala}, we get
$br_0(M)-br_1(M)=\ell (F/M)$. Therefore the result follows.
\end{proof}
As a consequence, we obtain a bound on the reduction number, similar
to that of \cite[Corollary 1.5]{rossi}. It may be noted that in
\cite{rossi}, the bound is derived
without the Cohen-Macaulay assumption on the fiber cone.
\begin{corollary}\label{rossi_bound}
Let $(R,\m)$ be a $2$-dimensional Cohen-Macaulay  local ring and
$M\subset F$ be such that $\ell(F/M) < \infty$ and having rank
$r$. Assume that the fiber cone $\F(M)$ is Cohen-Macaulay. Then
$$\red (M)\le br_1(M)-br_0(M)+\ell (F/M)+1.$$
\end{corollary}
\begin{proof}
    By Remark \ref{red_bound} and Theorem \ref {cpv-general},
  $$\red (M)\le f_0(M)-\mu (M)+r+1 \le br_1(M)-br_0(M)+\ell (F/M) + 1.$$
\end{proof}

In \cite{bala}, we obtained a Northcott type inequality for the
Buchsbaum-Rim coefficients and also proved that $\red(M)\le 1$ ensures
the equality. Now we prove a partial converse, i.e., the equality in
the Northcott inequality yields the reduction number to be at most one
under the assumption that the fiber cone is Cohen-Macaulay.
\begin{corollary}
Let $(R,\m)$ be a $2$-dimensional Cohen-Macaulay  local ring and
$M\subset F$ be such that $\ell(F/M) < \infty$ and having rank
$r$. Then the following are equivalent:
\begin{enumerate}
  \item $\F(M)$ is Cohen-Macaulay and $br_0(M)-br_1(M)=\ell (F/M)$,
  \item $\red(M) \leq 1$.
\end{enumerate}
\end{corollary}
\begin{proof}
$(1) \Rightarrow (2)$: Follows from Corollary \ref{rossi_bound}.\\
$(2) \Rightarrow (1)$: Follows from \cite[Theorem 3.3]{bala}.
\end{proof}

\section{Direct sum of ideals}

In this section, we study the Vasconcelos inequality for modules which
are direct sums of an $\m$-primary ideal. We begin by producing an
example to show that the inequality does not hold for modules over
1-dimensional rings.  Then we proceed to prove the result for
$d\ge 2$.

\begin{example}\cite[Example 6.2]{jayan-tony}
Let $k$ be a field and $R=k[\![t^7,t^{15},t^{17},t^{33}]\!]$, $I=(t^7,t^{17},t^{33})$ and
$J=(t^7)$. Then $R$ is a one dimensional Noetherian local domain and $I$
is an $\m$-primary ideal with minimal reduction $J$. Since $\ell
(I^2/JI)=1$ and $I^3=JI^2$, by \cite{sally}, $H_I(n)=P_I(n)$ for all
$n>1$, where $H_I(n)$ and $P_I(n)$ denote the Hilbert function and
Hilbert polynomial of $I$ respectively. It can be easily computed that
$P_I(n)=7n-5$. The fiber cone $\F(I)$ is Cohen-Macaulay \cite
[Theorem 3.4]{jayan-tony} and its multiplicity is $f_0(I)=4$. Let
$M=I\oplus I$. The Buchsbaum-Rim polynomial corresponding to
$M\subseteq F=R^2$ is given by
$$BP_M(n)=(n+1)(7n-5)=14{n+1\choose 2}-5{n\choose 1}-5.$$
Therefore
$$br_0(M)=14, br_1(M)=5, \ell (F/M)=6, \mu (M)=6, f_0(M)=f_0(I)=4.$$
Hence we have $f_0(M)=4>br_1(M)-br_0(M)+\ell(F/M)+\mu(M)-(d+r-2)=2$.
\end{example}
\begin{remark}\label{sum-formula}
Let $M=I\oplus \cdots \oplus I \subseteq F=R^r$. Let $S(F)\cong R[t_1,
\ldots, t_r]$ and $\R(M)\cong R[I t_1, \ldots, I t_r]$, where $t_1,
\ldots, t_r$ are indeterminates over $R$. The homogeneous
$R$-submodule $\R_n(M)$ of $\R(M)$ is given by $$\R_n(M)\cong
\sum_{i_1+\cdots +i_r=n} I^n t_{1}^{i_1}\cdots t_{r}^{i_r}.$$
Therefore, for $n\ge 0$, we have
  \begin{eqnarray*}
  \mu (\R_n(M)) &=& \ell \left(\frac{\R_n(M)}{\m \R_n(M)}\right) =
  {n+r-1\choose r-1}\mu (I^n)~ ~  ~ ~ ~\mbox{ and }\\
  BF_M(n) &=& \ell \left(\frac{\S_n(F)}{\R_n(M)}\right) = {n+r-1\choose r-1} \ell (R/I^n).
  \end{eqnarray*}
Hence for $n\gg 0$, $$\mu (\R_n(M))={n+r-1\choose r-1}
\sum_{i=0}^{d}(-1)^{i}f_i(I){n+d-1-i\choose d-1-i}.$$Therefore
$$f_0(M)={d+r-2\choose r-1}f_0(I).$$
Similarly we have
\begin{eqnarray*}
br_0(M)&=&{d+r-1\choose r-1}e_0(I),\\
br_1(M)&=&(d-1){d+r-2\choose r-2}e_0(I)+{d+r-2\choose r-1}e_1(I),\\
\ell (F/M)&=&r\ell (R/I),\mbox { and } \mu (M)=r\mu (I).
\end{eqnarray*}
\end{remark}

We conclude the article by presenting a class of modules for
which the Vasconcelos inequality holds true, namely, modules which are
direct sum of two $\m$-primary ideals $I$ and $J$, where one of them, say
$J$, is a reduction of $I$. In this case, it can be seen that the
fiber multiplicity $f_0(M)$ and the Buchsbaum-Rim coefficients
$br_0(M)$ and $br_1(M)$ depend only on $I$ and $r$, not on the
number of copies of $J$ involved in the direct sum. 
Recall that if $I$ is an $\m$-primary
ideal in a Cohen-Macaulay local ring $R$, then
$e_0(I) = \mu(I) + \ell(R/I) - d + \ell(\m I/\m J)$, \cite{Goto}. 
\begin{theorem}
Let $(R,\m)$ be a $d$-dimensional Cohen-Macaulay local ring with $d\ge
2$. Let $I$ be an $\m$-primary ideal and $J$ be a reduction. Let
$\mathbf{I}=I\oplus \cdots \oplus I$(u-times), $\mathbf{J}=J\oplus
\cdots \oplus J$(v-times) and $M=\mathbf{I}\oplus \mathbf {J}\subset
F = R^r,$ where $r=u+v$. Then $$f_0(M)\le br_1(M)-br_0(M)+\ell (F/M) +\mu
(M)-(d+r-2).$$
\end{theorem}
\begin{proof}
Since the assertion is proved for $d = 2$ in the previous section, we
may assume that $d \geq 3$.
Let us assume that result holds for $M'=I\oplus \cdots \oplus I (r-times)$.
  Since $J$ is a reduction of $I$, $JI^s=I^{s+1}$ for some $s\in \mathbb{N}$ and $e_0(I)=e_0(J)$. Let
$$\Delta = \sum _{i=s}^n {i+u-1\choose u-1}{n-i+v-1\choose v-1} ~ and ~\delta = \sum _{i=0}^{s-1} {i+u-1\choose u-1}{n-i+v-1\choose v-1}.$$
Note that $\Delta = {n+r-1\choose r-1}-\delta.$ The Buchsbaum-Rim function is given by

\begin{eqnarray*}
BF(n)&=&\Delta  \ell (R/I^n) + \sum _{i=0}^{s-1} \left[{i+u-1\choose u-1}{n-i+v-1\choose v-1}\ell (R/I^iJ^{n-i})\right]\\
&=& \left[{n+r-1\choose r-1}-\delta \right]\ell (R/I^n)\\
& & + \sum _{i=0}^{s-1} \left[{i+u-1\choose u-1}{n-i+v-1\choose v-1}\ell (R/I^iJ^{n-i})\right]\\
&=& {n+r-1\choose r-1}\ell (R/I^n) + \sum _{i=0}^{s-1} \left[{i+u-1\choose u-1}{n-i+v-1\choose v-1}\ell (I^n /I^iJ^{n-i})\right].
\end{eqnarray*}
Note for each $i=0,\ldots,s-1,$
$$\ell \left(\frac{R}{I^n}\right)\le \ell \left(\frac {R}{I^iJ^{n-i}}\right) \le \ell \left(\frac {R}{J^n}\right).$$
Hence for each fixed $i$ and $n\gg 0$, the function $\ell (R/I^iJ^{n-i})$ is given by polynomial of degree $d$ with leading coefficient $e_0(I)$.
This implies $\ell (I^n /I^iJ^{n-i})=\ell (R/I^iJ^{n-i}) - \ell (R/I^n)$ is given by polynomial of degree at most $d-1$ for large $n$.\\
By considering $n\gg 0$, the Buchsbaum-Rim polynomial of $M$ is given by
$$BP(n)={n+r-1\choose r-1}P_I(n)+ O(n^{d+r-3}).$$
Note that the first term in the above expression is the Buchsbaum-Rim polynomial of $M'=I\oplus \cdots\oplus I$($r$-times). Therefore
$br_0(M)=br_0(M')$ and $br_1(M)=br_1(M')$.
  Similarly we show that $f_0(M)=f_0(M')$. Let $s=\red _J(I)$. Then $JI^i=I^{i+1}$ for $i\ge s$.
  \begin{eqnarray*}
    \mu (\R_n(M))&=&\left[{n+r-1\choose r-1}-\sum _{i=0}^{s-1} \left[{i+u-1\choose u-1}{n-i+v-1\choose v-1}\right]\right]\mu (I^n)\\
     &~&~~~~~~~~~~ + \sum _{i=0}^{s-1} \left[{i+u-1\choose u-1}{n-i+v-1\choose v-1}\ell \left(\frac{I^iJ^{n-i}}{\m I^iJ^{n-i}}\right)\right]\\
    &=& {n+r-1\choose r-1}\mu (I^n)\\
    &~&~~~~~~~~~~+\sum _{i=0}^{s-1} {i+u-1\choose u-1}{n-i+v-1\choose v-1}\left[\ell \left(\frac{I^iJ^{n-i}}{\m I^iJ^{n-i}}\right)-\ell \left(\frac{I^n}{\m I^n}\right)\right]\\
    &=&{d+r-2\choose r-1}f_0(I){n+d+r-2\choose d+r-2} + O(n^{d+r-3}).
  \end{eqnarray*}
  Hence $f_0(M)={d+r-2\choose r-1}f_0(I)=f_0(M')$.  Also note that
  
\begin{eqnarray*}
\ell (F/M)&=&u\ell (R/I)+v\ell (R/J)=\ell (F/M')-v\ell (R/I)+v\ell (R/J)\\
\mbox{and }\mu (M)&=&u\mu (I)+v\mu (J)=\mu (M')-v\mu (I)+v\mu (J).
\end{eqnarray*}
Therefore
  \begin{eqnarray*}
     br_1(M)&-&br_0(M)+\ell (F/M)+\mu (M)-(d+r-2)-f_0(M)\\
    &=&br_1(M')-br_0(M')+\ell (F/M')+\mu (M')-(d+r-2)-f_0(M')\\
    &~&~~~~~~~~+v[\ell (R/J)+\mu (J)-(\ell (R/I)+\mu (I))]\\
    &\ge &0.
  \end{eqnarray*}
  It remains to prove that Vasconcelos inequality holds for $M'=I\oplus \cdots \oplus I (r-times)$.
  Set $\Lambda = br_1(M') - br_0(M') +\ell
  (F/M')+\mu(M')-(d+r-2)-f_0(M')$ and $e_i = e_i(I)$ for $i =
  0,1,\ldots, d$. Then from Remark \ref{sum-formula}, it follows that

  \begin{eqnarray}\label{Lambda}
  \Lambda &=&(d-1){d+r-2\choose r-2}e_0+{d+r-2\choose
	r-1}e_1-{d+r-1\choose r-1}e_0+r[\ell (R/I)+\mu(I)] \nonumber\\
    &~&~~~~-(d+r-2)-{d+r-2\choose r-1}f_0(I)\\
    &\ge& (d-1){d+r-2\choose r-2}e_0+{d+r-2\choose
	r-1}e_1-{d+r-1\choose r-1}e_0+r[\ell (R/I)+\mu(I)]\nonumber\\
    &~&~~~~-(d+r-2)-{d+r-2\choose r-1}
	[e_1-e_0+\ell(R/I)+\mu(I)-(d-1)] \nonumber \\
    &=& (d-1){d+r-2\choose r-2}e_0+{d+r-2\choose
	r-1}e_1-\left[{d+r-2\choose r-1}+{d+r-2\choose r-2}\right]e_0
	\nonumber \\
    &~&+r[\ell (R/I)+\mu(I)]-(d+r-2)
    -{d+r-2\choose r-1} [e_1-e_0+\ell(R/I)+\mu(I)-(d-1)] \nonumber\\
	&=&\left[(r-1)\frac{(d-2)}{d}e_0-\ell(R/I)-\mu(I)+d-1\right]{d+r-2\choose
	r-1} \nonumber\\
	& & +r[\ell(R/I)+\mu(I)]-(d+r-2). \nonumber 
  \end{eqnarray}
If $r = 2$, then the last expression reduces to
\[
\Lambda \ge (d-2)[e_0-\ell (R/I)-\mu (I)+d] \geq 0.
  \]
If $r = 3$, then we have
\begin{eqnarray*}
  \Lambda & \geq & \left[\frac{2(d-2)}{d}e_0 - \ell(R/I) - \mu(I) + d
	-1\right]{d+1 \choose 2} + 3[\ell(R/I) + \mu(I)] - (d+1) \\
	& = & \left[\frac{d-4}{d}e_0 + \ell\left(\frac{\m I}{\m J}\right)
	  - 1\right]{d+1 \choose 2} + 3 [e_0 + d - \ell(\m I/\m J)] -
	  (d+1),
\end{eqnarray*}
where the last equality holds since $e_0
 = \ell(R/I) + \mu(I) - d+ \ell(\m I/\m J)$. Now splitting the proof
into the cases
$d=3$ and $d \geq 4$ together with $\ell(\m I/\m J) = 0$ or $\ell(\m
I/\m J) > 0$, one can easily obtain the inequality $\Lambda \geq 0$.

\vskip 2mm
\noindent
If $r \geq 4$, then $\frac{(r-1)(d-2)}{d} \geq 1$ and the equality
holds if and only if $r = 4$ and $d = 3$.  It is
straightforward to verify the inequality $\Lambda \geq 0$ if $d = 3$
and $r = 4$. If $d > 3$ or $r > 4$, then one obtains the required
inequality by splitting the proof into the cases $\ell(\m I/\m J) = 0$
and $\ell(\m I/\m J) \geq 1$.
\end{proof}
\begin{remark}
Suppose $I$ is a parameter ideal. Then $e_0(I) = \ell(R/I), e_1(I) =
0, \mu(I) = d$ and $f_0(I) = 1$. Therefore, it can be seen from the
expression (\ref{Lambda}) that 
\begin{enumerate}
  \item if $r = 2$, then Vasconcelos inequality becomes an equality,
  \item if $r = 3$, the equality holds if and only if $e_0(I) = 1$
	which happens if and only if $R$ is a regular local ring and $I$
	is the unique maximal ideal; 
  \item if $r \geq 4$, the equality never holds.
\end{enumerate}

\end{remark}

\end{document}